\documentclass[12pt,a4paper]{article}
\usepackage{amsfonts}
\usepackage{amssymb}
\usepackage{mathrsfs}
\usepackage{amsmath}
\usepackage{amsthm}
\usepackage{enumerate}
\usepackage{cite}
\newtheorem{defn}{Definition}[section]
\newtheorem{thm}[defn]{Theorem}
\newtheorem{lem}[defn]{Lemma}
\newtheorem{prop}[defn]{Proposition}
\newtheorem{cor}[defn]{Corollary}
\newtheorem{ex}[defn]{Example}
\newtheorem{re}[defn]{Remark}
\def\K{{\bf K}}

\def\ad{{{\rm ad}}}
\def\dim{{{\rm dim}}}
\def\ch{{{\rm ch}}}

\begin{document}
\title{{\bf Hom-Nijienhuis operator and $T$*-extension of Hom-Lie Superalgebras}}
\author{\normalsize \bf Yan Liu,  Liangyun Chen, Yao Ma}
\date{{{\small{ School of Mathematics and Statistics,\\
  Northeast Normal University,\\
   Changchun 130024, China
 }}}} \maketitle
\date{}

 {\bf\begin{center}{Abstract}\end{center}}

In this paper, we  study  hom-Lie superalgebras. We give the
definition of hom-Nijienhuis operators of regualr hom-Lie
superalgebras and show that the deformation generated by a
hom-Nijienhuis operator is trivial. Moreover, we introduce the
definition of $T^*$-extensions of  Hom-Lie superalgebras and show
that $T^*$-extensions preserve many properties such as nilpotency,
solvability and decomposition in some sense. We also investigate the
equivalence of $T^*$-extensions.

\renewcommand{\thefootnote}{\fnsymbol{footnote}}
\footnote[0]{ Corresponding author(L. Chen): chenly640@nenu.edu.cn.}
\footnote[0]{ Supported by  NNSF of China (No. 11171057),  Natural
Science Foundation of  Jilin province (No. 201115006), Scientific
Research Foundation for Returned Scholars
    Ministry of Education of China and the Fundamental Research Funds for the Central Universities. }

\section{Introduction}

A Hom-Lie algebra is a triple $(L,[\cdot,\cdot]_{L},\alpha)$, where $\alpha$ is a linear self-map, in which the skewsymmetric bracket satisfies an $\alpha$-twisted variant of the Jacobi identity, called the Hom-Jacobi
identity. When $\alpha$ is the identity map, the Hom-Jacobi identity reduces to the usual Jacobi identity, and $L$ is a Lie algebra. The notion of  hom-Lie algebras was introduced by Hartwig, Larsson and Silvestrov to describe the structures on certain deformations of the
Witt algebra and the Virasoro algebra \cite {Hartwig&Larsson}. Hom-Lie algebras are also related to deformed vector fields, the various versions of the Yang-Baxter equations, braid group representations, and quantum groups[7,12,13]. Recently, hom-Lie algebras were generalized to hom-Lie superalgebras by Ammar and Makhlouf [1-3].  More applications of the hom-Lie algebras, hom-algebras and hom-Lie superalgebras can be found in [6-8, 11-13]

In this paper we introduce an extensional technique called
$T^*$-extensional for hom-Lie superalgebra. This method is a
one-step procedure, so it is usually workable. In 1997, Bordemann
introduced the notion of $T^*$-extension of Lie algebras(see
\cite{Bordemann}), which is one of the main tools to prove that
every symplectic  Manin quadratic Lie algebra is a special
symplectic algebra in \cite{Bajo&Benayadi&Medina}. Many facts show
that $T^*$-extension is an important method to study algebraic
structures [4, 5, 9]

The paper proceeds as follows. In Section 2 after giving the definition of hom-Lie superalgebras,  we show that the direct sum of two hom-Lie superalgebras
is still a hom-Lie superalgebra. A linear map between hom-Lie superalgebras  is a morphism if and only if its graph is a hom-Lie sub-superalgebra. Section 3
we give the definition of hom-Nijienhuis operators of regualr hom-Lie superalgebras. We show that the deformation generated by a hom-Nijienhuis operator is
trivial. Section 4 after presents a summary of some of the relevant concepts, we introduces the definition of $T^*$-extension of hom-Lie superalgebra and
shows that $T^*$-extension preserves many properties such as nilpotency, solvability and decomposition in some sense. Section 4 discusses the equivalence
of $T^*$-extensions using cohomology.

\section{Preliminaries}
\begin{defn}
(1) A hom-Lie superalgebra is a triple  $(L,[\cdot,\cdot]_{L},\alpha)$ consisting of a $Z_2$-graded vector space $L$, an even bilinear map (bracket)
$[\cdot,\cdot]_L:{\wedge}^2L\rightarrow L$ and an even homomorphism $\alpha:L\rightarrow L$ satisfying
\begin{equation}[x,y]_{L}=-(-1)^{|x||y|}[y,x]_{L},\end{equation}
\begin{equation}(-1)^{|x||z|}[\alpha(x),[y,z]_{L}]_{L}+(-1)^{|y||x|}[\alpha(y),[z,x]_{L}]_{L}+(-1)^{|z||y|}[\alpha(z),[x,y]_{L}]_{L}=0,
\end{equation}
where $x,y$ and $z$ are homogeneous elements in $L$;

(2) A hom-Lie superalgebra is called a multiplicative hom-lie superalgebra if $\alpha$ is an algebraic morphism, i.e. for any $x, y, \in L$, we have
 $\alpha([x,y]_L)=[\alpha(x),\alpha(y)]_L;$

(3) A hom-Lie superalgebra is called a regular hom-lie superalgebra if $\alpha$ is an algebra automorphism;

(4) The center of hom-Lie superalgebras $L$, denoted by $Z(L)$, is the set of such elements $x\in L$ satisfying $[x,L]=0$;

(5) A sub-vector space $I\subset L$ is a hom-Lie subsuperalgebra  of $(L,[\cdot,\cdot]_{L},\alpha)$  if $\alpha(I)\subset I$  and $I$ is closed under
 the bracket operation  $[\cdot,\cdot]_L$, i.e.
$[I,I]_L\in I$;

(6) $I$ is called a hom-Lie superalgebra ideal of $L$  if
$[I,L]_L\subset I$. Moreover, if $[I,I]=0$, then $I$ is called an
abelian hom-Lie superalgebra ideal of $L$.
 Note that a subspace $L^{'}$ of $L$ has a  natural $Z_{2}$-gradation: $L^{'}=(L^{'}\cap L_{\bar{0}})\oplus(L^{'}\cap L_{\bar{1}})$.
\end{defn}

In the sequel, $u+v\in L\oplus \Gamma$ implies that $u\in L$, $v\in \Gamma$ and $u+v$ is homogeneous satisfying $|u+v|=|u|=|v|$, then $(-1)^{|u_1+v_1||u_2+v_2|}=(-1)^{|u_1||u_2|}$.
\begin{prop}\label{proposition2.1}
Given two hom-Lie superalgebras $(L,[\cdot,\cdot]_L,\alpha)$ and $(\Gamma,[\cdot,\cdot]_\Gamma,\beta)$, there is a  hom-Lie superalgebra
$(L\oplus\Gamma,[\cdot,\cdot]_{L\oplus\Gamma},\alpha+\beta)$, where the bilinear map $[\cdot,\cdot]_{L\oplus\Gamma}:{\wedge}^2(L\oplus\Gamma)\rightarrow L\oplus\Gamma$ is given by
$${[u_1+v_1,u_2+v_2]}_{L\oplus\Gamma}={[u_1,u_2]}_L+{[v_1,v_2]}_\Gamma, \qquad \forall  u_1,u_2\in L,  v_1,v_2 \in \Gamma, $$
and the linear map $(\alpha+\beta):L\oplus\Gamma \rightarrow L\oplus\Gamma$ is given by
$$(\alpha+\beta)(u+v)=\alpha(u)+\beta(v), \qquad \forall  u\in L, v\in \Gamma.$$
 \end{prop}
\begin{proof}
 for any  $u_i\in L,  v_i\in \Gamma,$ we have
$${[u_1+v_1,u_2+v_2]}_{L\oplus\Gamma}={[u_1,u_2]}_L+{[v_1,v_2]}_\Gamma,$$
\begin{eqnarray*}
-(-1)^{|u_1||u_2|}{[u_2+v_2,u_1+v_1]}_{L\oplus\Gamma}&=&-(-1)^{|u_1||u_2|}({[u_2,u_1]}_L+{[v_2,v_1]}_\Gamma)\\
&=&{[u_1,u_2]}_L+{[v_1,v_2]}_\Gamma,
\end{eqnarray*}
The bracket is obviously supersymmetric and with a direct computation we have
\begin{eqnarray*}
&&(-1)^{|u_{1}||u_{3}|}[(\alpha+\beta)(u_1+v_1),[u_2+v_2,u_3+v_3]_{L\oplus\Gamma}]_{L\oplus\Gamma}+c.p.((u_1+v_1),(u_2+v_2),(u_3+v_3))\\
&=&(-1)^{|u_{1}||u_{3}|}[\alpha(u_1)+\beta(v_1),[u_2,u_3]_{L}+[v_2,v_3]_{\Gamma}]_{L\oplus\Gamma}+c.p.\\
&=&(-1)^{|u_{1}||u_{3}|}[\alpha(u_1),[u_2,u_3]_L]_L+c.p.(u_1,u_2,u_3)+(-1)^{|v_{1}||v_{3}|}[\beta(v_1),[v_2,v_3]_\Gamma]_\Gamma+c.p.(v_1,v_2,v_3)\\
&=&0.
\end{eqnarray*}
where c.p.(a,b,c) means the cyclic permutations of a, b, c.
\end{proof}
\begin{defn}Let $(L,[\cdot,\cdot]_L,\alpha)$ and $(\Gamma,[\cdot,\cdot]_\Gamma,\beta)$ be two hom-Lie superalgebras. An even homomorphism $\phi:L \rightarrow \Gamma$ is said to be a morphism of Hom-lie superalgebras if \\
\begin{equation} \phi[u,v]_L=[\phi(u),\phi(v)]_\Gamma,\qquad\forall  u,v\in L,\end{equation}
\begin{equation}\phi\circ\alpha=\alpha\circ\phi.\qquad\qquad\qquad\qquad\end{equation}

 Denote by $\mathfrak{G}_\phi\in L\oplus\Gamma$ the graph of a linear map $\phi:L \rightarrow \Gamma$.
\end{defn}
\begin{prop}\label{proposition2.1}
An even homomorphism $\phi:(L,[\cdot,\cdot]_L,\alpha)\rightarrow(\Gamma,[\cdot,\cdot]_\Gamma,\beta)$ is a morphism of hom-Lie superalgebras if
and only if the graph $\mathfrak{G}_\phi\in L\oplus\Gamma$ is a hom-Lie sub-superalgebra of  $(L\oplus\Gamma,[\cdot,\cdot]_{L\oplus\Gamma},\alpha+\beta)$.
\end{prop}
\begin{proof}
Let
$\phi:(L,[\cdot,\cdot]_L,\alpha)\rightarrow(\Gamma,[\cdot,\cdot]_\Gamma,\beta)$
be a morphism of hom-Lie superalgebras, then for any $u,v\in L$, we have
$$[u+\phi(u),v+\phi(v)]_{L\oplus\Gamma}=[u,v]_{L}+[\phi(u),\phi(v)]_{\Gamma}=[u,v]_{L}+\phi[u,v]_L.$$
Thus the graph $\mathfrak{G}_\phi$ is closed under the bracket operation $[\cdot,\cdot]_{L\oplus\Gamma}$. Furthermore, by (4), we have
$$(\alpha+\beta)(u+\phi(u))=\alpha(u)+\beta\circ\phi(u)=\alpha(u)+\phi\circ\alpha(u)$$
which implies that $(\alpha+\beta)(\mathfrak{G}_\phi)\subset \mathfrak{G}_\phi$. Thus $\mathfrak{G}_\phi$ is a hom-Lie sub-superalgebra of $(L\oplus\Gamma,[\cdot,\cdot]_{L\oplus\Gamma},\alpha+\beta)$.

Conversely, if the graph $\mathfrak{G}_\phi\subset L\oplus\Gamma$ is a hom-Lie sub-superalgebra of
$(L\oplus\Gamma,[\cdot,\cdot]_{L\oplus\Gamma},\alpha+\beta)$, then we have
$$[u+\phi(u),v+\phi(v)]_{L\oplus\Gamma}=[u,v]_{L}+[\phi(u),\phi(v)]_{\Gamma}\in\mathfrak{G}_\phi,$$
which implies that $$[\phi(u),\phi(v)]_{\Gamma}=\phi[u,v]_{L}.$$
Furthermore, $(\alpha+\beta)(\mathfrak{G}_{\phi})\subset \mathfrak{G}_{\phi}$ yields that
$$(\alpha+\beta)(u+\phi(u))=\alpha(u)+\beta\circ\phi(u)\in\mathfrak{G}_\phi,$$
which is equivalent to the condition $\beta\circ\phi(u)=\phi\circ\alpha(u)$, i.e. $\beta\circ\phi=\phi\circ\alpha.$ Therefore, $\mathfrak{G}_{\phi}$ is a morphism of hom-Lie superalgebras.
\end{proof}

\section{The hom-Nijienhuis operator of hom-Lie superalgebras}
Let $(L,[\cdot,\cdot]_{L},\alpha)$ be a multiplicative hom-Lie superalgebra. We consider that L represents on itself via the bracket with respect to the morphism $\alpha$.
\begin{defn}\cite{Faouzi&Abdenacer}
For any integer s, the $\alpha^{s}$-adjoint representation of the multiplicative  hom-Lie superalgebra $(L,[\cdot,\cdot]_{L},\alpha)$, which we denote by $ad_{s}$, is defined by
$$\ad_{s}(u)(v)=[\alpha^{s}(u),v]_{L},\forall u,v\in L$$
In particular we use $\ad$ represent $\ad_{0}$.
\end{defn}
\begin{lem}\cite{Faouzi&Abdenacer}
With the above notations, we have
$$\ad_{s}(\alpha(u))\circ \alpha=\alpha\circ \ad_{s}(u);$$
$$\ad_{s}([u,v]_{L})\circ \alpha=\ad_{s}(\alpha(u))\circ \ad_{s}(v)-(-1)^{|u||v|}\ad_{s}(\alpha(v))\circ \ad_{s}(u).$$
Thus the definition of $\alpha^{s}$-adjoint representation is well defined.
\end{lem}
The set of k-hom-cochains on L with coefficients in L,  which we denote by $C^{k}_{\alpha}(L;L)$ is given by
$$C^{k}_{\alpha}(L;L)=\{f\in C^{k}(L;L)|\alpha \circ f=f\circ \alpha\}.$$
In particular,  the set of 0-hom-cochains are given by:
$$C^{0}_{\alpha}(L;L)=\{u\in L|\alpha(u)=u\}.$$

Associated to the $\alpha ^{s}$-adjoint representation,  the coboundary operator $d_s:C_{\alpha}^{k}(L;L)\rightarrow C_{\alpha}^{k+1}(L;L)$ is given by
\begin{eqnarray*}
d_{s}f(u_{0}, \cdots, u_{k})&=&\sum_{i=0}^k(-1)^{i}(-1)^{(|f|+|u_{0}|+\cdots +|u_{i-1}|)|u_{i}|}[\alpha^{k+s}(u_{i}),f(u_{0}, \cdots, \hat{{u_{i}}} \cdots, u_{k})]_{L}\\
&+&\sum_{i<j}(-1)^{i+j}(-1)^{(|u_{0}|+\cdots +|u_{i-1}|)|u_{i}|}(-1)^{(|u_{0}|+\cdots +|u_{j-1}|)|u_{j}|}(-1)^{|u_{i}||u_{j}|}\\
&&f([u_{i}, u_{j}], \alpha(u_{0}), \cdots, \hat{\alpha(u_{i})}, \cdots, \hat{\alpha(u_{i})}, \cdots, \alpha(u_{k}))
\end{eqnarray*}
For the $\alpha^{s}$-adjoint representation $\ad_{s}$,  we obtain the $\alpha ^{s}$-adjoint complex $(C_{\alpha}^{\bullet}(L;L), d_{s})$ and the corresponding cohomology
$$H^{k}(L;\ad_{s})=Z^{k}(L;\ad_{s})/B^{k}(L;\ad_{s}).$$
Let $\psi\in C_{\alpha}^{2}(L;L)$  be a bilinear operator commuting whit $\alpha$. Consider a t-parametrized family of bilinear operations
\begin{eqnarray}
[u, v]_{t}=[u, v]_{L}+t\psi(u, v).
\end{eqnarray}
Since $\psi$ commutes with $\alpha$,  $\alpha$ is a morphism with respect to the bracket $[\cdot, \cdot]_{t}$ for every $t$. If all the brackets $[\cdot, \cdot]_{t}$ endow $(L, [\cdot, \cdot]_{t}, \alpha)$ regular hom-Lie superalgebra structures,  we say that $\psi$ generates a deformation of the regular hom-Lie superalgebra $(L, [\cdot, \cdot]_{L}, \alpha)$. By computing the hom-superJacobi identity of $[\cdot, \cdot]_{t}$, this is equivalent to the conditions
\begin{eqnarray}(-1)^{|u||w|}(\psi(\alpha(u), \psi[v, w]))+c.p.(u, v, w)=0;\\
(-1)^{|u||w|}(\psi(\alpha(u), [v, w]_{L})+[\alpha(u), \psi[v, w]_{L}]_{L})+c.p.(u, v, w)=0.
\end{eqnarray}
Obviously,   (6) means that $\psi$ must itself define a hom-Lie superalgebra structure on L. Furthermore,  (7) means that $\psi$ is closed with respect to the $\alpha^{-1}$-adjoint representation $\ad_{-1}$,  i.e. $d_{-1}\psi=0$.

A deformation is  said to be trivial if there is a linear operator $N\in C_{\alpha}^{1}(L;L)$ such that for $T_{t}=Id+tN$,  there holds
\begin{eqnarray}
T_{t}[u, v]_{t}=[T_{t}(u), T_{t}(v)]_{L}.
\end{eqnarray}
\begin{defn}\cite {Sheng&Yunhe}
A linear operator $N\in  C_{\alpha}^{1}(L, L)$ is called a hom-Nijienhuis operator if we have
\begin{eqnarray}
[Nu, Nv]_{L}=N[u, v]_{N},
\end{eqnarray}
where the bracket $[\cdot, \cdot]_{N}$ is defined by
\begin{eqnarray}
[u, v]_{N}\triangleq[Nu, v]_{L}+[u, Nv]_{L}-N[u, v]_{L}.
\end{eqnarray}
\end{defn}
\begin{thm}
Let $N\in C_{\alpha}^{1}(L, L)$ be a hom-Nijienhuis operator. Then a deformation of the regular hom-Lie superalgebra $(L, [\cdot, \cdot]_{L}, \alpha)$ can be obtained by putting
$$\psi(u, v)=d_{-1}N(u, v)=[u, v]_{N}.$$
Furthermore,  this deformation is trivial.
\end{thm}
\begin{proof}
Since $\psi=d_{-1}N$,  $d_{-1}\psi=0$ is valid. To see that $\psi$ generates a deformation,  we need to check the hom-superJacobi identity for $\psi$. Using the explicit expression of  $\psi$,  we have
\begin{eqnarray*}&&(-1)^{|u||w|}\psi(\alpha(u), \psi(v, w))+c.p.(u, v, w)\\
&=&(-1)^{|u||w|}[\alpha(u), [v, w]_{N}]_{N}+c.p.(u, v, w)\\
&=&(-1)^{|u||w|}([\alpha(u), [Nv, w]_{L}+[v, Nw]_{L}-N[v, w]_{L}]_{N})+c.p.(u, v, w)\\
&=&(-1)^{|u||w|}([\alpha(u), [Nv, w]_{L}]_{N}+[\alpha(u), [v, Nw]_{L}]_{N}-[\alpha(u), N[v, w]_{L}]_{N})+c.p.(u, v, w)\\
&=&(-1)^{|u||w|}([N\alpha(u), [Nv, w]_{L}]_{L}+[\alpha(u), N[Nv, w]_{L}]_{L}-N[\alpha(u), [Nv, w]_{L}]_{L}\\
&&+[N\alpha(u), [v, Nw]_{L}]_{L}+[\alpha(u), N[v, Nw]_{L}]_{L}-N[\alpha(u), [v, Nw]_{L}]_{L}\\
&&-[N\alpha(u), N[v, w]_{L}]_{L}-[\alpha(u), N^{2}[v, w]_{L}]_{L}+N[\alpha(u), N[v, w]_{L}]_{L})+c.p.(u, v, w).
\end{eqnarray*}
Since
\begin{eqnarray*}&&[\alpha(u), N[Nv, w]_{L}]_{L}+[\alpha(u), N[v, Nw]_{L}]_{L}-[\alpha(u), N^{2}[v, w]_{L}]_{L}\\
&=&[\alpha(u), N([Nv, w]_{L}+[v, Nw]_{L}-N[v, w]_{L})]_{L}\\
&=&[\alpha(u), N[v, w]_{N}]_{L}.
\end{eqnarray*}
Therefore,  we have
\begin{eqnarray*}&&(-1)^{|u||w|}\psi(\alpha(u), \psi(v, w))+c.p.(u, v, w)\\
&=&(-1)^{|u||w|}([N\alpha(u), [Nv, w]_{L}]_{L}-N[\alpha(u), [Nv, w]_{L}]_{L}+[N\alpha(u), [v, Nw]_{L}]_{L}\\
&&-N[\alpha(u), [v, Nw]_{L}]_{L}-[N\alpha(u), N[v, w]_{L}]_{L}+N[\alpha(u), N[v, w]_{L}]_{L}+[\alpha(u), N[v, w]_{N}]_{L})\\
&&+(-1)^{|u||v|}([N\alpha(v), [Nw, u]_{L}]_{L}-N[\alpha(v), [Nw, u]_{L}]_{L}+[N\alpha(v), [w, Nu]_{L}]_{L}\\
&&-N[\alpha(v), [w, Nu]_{L}]_{L}-[N\alpha(v), N[w, u]_{L}]_{L}+N[\alpha(v), N[w, u]_{L}]_{L}+[\alpha(v), N[w, u]_{N}]_{L})\\
&&+(-1)^{|v||w|}([N\alpha(w), [Nu, v]_{L}]_{L}-N[\alpha(w), [Nu, v]_{L}]_{L}+[N\alpha(w), [u, Nv]_{L}]_{L}\\
&&-N[\alpha(w), [u, Nv]_{L}]_{L}-[N\alpha(w), N[u, v]_{L}]_{L}+N[\alpha(w), N[u, v]_{L}]_{L}+[\alpha(w), N[u, v]_{N}]_{L})\\
&=&(-1)^{|u||w|}[N\alpha(u), [Nv, w]_{L}]_{L}+(-1)^{|u||v|}[N\alpha(v), [w,Nu]_{L}]_{L}\\
&&+(-1)^{|v||w|}[\alpha(w), N([u, v]_{N})]_{L}+c.p.(u, v, w)\\
&&+(-1)^{|u||v|}(N[\alpha(v), N[w, u]_{L}]_{L}-N[\alpha(v),[w, u]_{L}]_{L})+c.p.(u, v, w)\\
&&-N((-1)^{|u||w|}[\alpha(u), [Nv, w]_{L}]_{L}+(-1)^{|v||w|}[\alpha(w), [u, Nv]_{L}]_{L})+c.p.(u, v, w).
\end{eqnarray*}
Since N commutes with $\alpha$,  by the hom-superJacobi identity of L,  we have
$$(-1)^{|u||w|}[N\alpha(u), [Nv, w]_{L}]_{L}+(-1)^{|u||v|}[\alpha(Nv), [w, Nu]_{L}]_{L}+(-1)^{|v||w|}[\alpha(w), [Nu, Nv]_{L}]_{L}=0.$$
Since N is a hom-Nijienhuis operator,  we have
\begin{eqnarray*}
&&(-1)^{|u||w|}[N\alpha(u), [Nv, w]_{L}]_{L}+(-1)^{|u||v|}[\alpha(Nv), [w, Nu]_{L}]_{L}+(-1)^{|v||w|}[\alpha(w), [u, v]_{N}]_{L}\\
&&+c.p.(u, v, w)=0.
\end{eqnarray*}
Furthermore,  also by the fact that $N$ is a hom-Nijienhuis
operator, we have
\begin{eqnarray*}
&&(-1)^{|u||v|}(N[\alpha(v), N[w, u]_{L}]_{L}-N[\alpha(v),[w, u]_{L}]_{L})+c.p.(u, v, w)\\
&=&(-1)^{|u||v|}(-N[N\alpha(v), [w, u]_{L}]_{L}+N^{2}[\alpha(v), [w, u]_{L}]_{L})+c.p.(u, v, w)\\
&=&-(-1)^{|u||v|}N[N\alpha(v), [w, u]_{L}]_{L}+(-1)^{|u||v|}N^{2}[\alpha(v), [w, u]_{L}]_{L})+c.p.(u, v, w).
\end{eqnarray*}
Thus by the hom-superJacobi identity of $L$,  we have
\begin{eqnarray*}
&&(-1)^{|u||v|}(N[\alpha(v), N[w, u]_{L}]_{L}-N[\alpha(v), [Nw, u]_{L}]_{L})+c.p.(u, v, w)\\
&=&-(-1)^{|u||v|}N[N\alpha(v), [w, u]_{L}]_{L}+c.p.(u, v, w).
\end{eqnarray*}
Therefore,  we have
\begin{eqnarray*}&&(-1)^{|u||w|}\psi(\alpha(u), \psi(v, w))+c.p.(u, v, w)\\
&=&-(-1)^{|u||v|}N[N\alpha(v), [w, u]_{L}]_{L}-N((-1)^{|u||w|}[\alpha(u), [Nv, w]_{L}]_{L}\\
&&+(-1)^{|v||w|}[\alpha(w), [u, Nv]_{L}]_{L})+c.p.(u, v, w)\\
&=&-N((-1)^{|u||v|}[\alpha(Nv), [w, u]_{L}]_{L}+(-1)^{|u||w|}[\alpha(u), [Nv, w]_{L}]_{L}\\
&&+(-1)^{|v||w|}[\alpha(w), [u, Nv]_{L}]_{L})+c.p.(u, v, w)\\
&=&0.
\end{eqnarray*}
Thus $\psi$ generates a deformation of the hom-Lie superalgebra $(L, [\cdot, \cdot]_{L}, \alpha)$.

Let $T_{t}=Id+tN$,  then we have
\begin{eqnarray*}
T_{t}[u, v]_{t}&=&(Id+tN)([u, v]_{L}+t\psi(u, v))\\
&=&(Id+tN)([u, v]_{L}+t[u, v]_{N})\\
&=&[u, v]_{L}+t([u, v]_{N}+N[u, v]_{L})+t^{2}N[u, v]_{N}.
\end{eqnarray*}
On the other hand,  we have
\begin{eqnarray*}
[T_{t}(u), T_{t}(v)]_{L}&=&[u+tNu, v+tNv]_{L}\\
&=&[u, v]_{L}+t([Nu, v]_{N}+[u, Nv]_{L})+t^{2}[Nu, Nv]_{L}.
\end{eqnarray*}
By (9), (10),  we have
$$T_{t}[u, v]_{t}=[T_{t}(u), T_{t}(v)]_{L}, $$
which implies that the deformation is trivial.
\end{proof}

\section{$T$*-extension of hom-Lie superalgebras}
\begin{defn}
Let $(L,[\cdot,\cdot]_L,\alpha)$ be a hom-Lie superalgebras. A bilinear form  $f$ on $L$ is said to be nondegenerate if
$$L^\perp=\{x\in L|f(x,y)=0, \forall y\in L\}=0,$$
invariant if
$$f([x,y],z)=f(x,[y,z]), \forall x,y,z\in L,$$
supersymmetric if
$$f(x,y)=-(-1)^{|x||y|}f(y,x),$$
A subspace $I$ of $L$ is called isotropic if $I\subseteq I^\bot$.
\end{defn}
\begin{defn}
A bilinear form $f$ on a hom-Lie superalgebras $(L,[\cdot,\cdot]_L,\alpha)$ is said to be superconsistent if $f$ satisfies
$$f(x,y)=0, \forall x\in L_{|x|}, y\in L_{|y|}, |x|+|y|\neq0.$$
Throughout this paper, we only consider superconsistent bilinear forms.
\end{defn}
\begin{defn}
Let $(L,[\cdot,\cdot]_L,\alpha)$ be a hom-Lie superalgebras over a field $\K$. If $L$ admits a nondegenerate invariant supersymmetric bilinear form $f$, then we call $(L,f,\alpha)$ a quadratic hom-Lie superalgebras. In particular, a quadratic vector space $V$ is a $Z_{2}$-graded vector space admitting a nondegenerate supersymmetric bilinear form.

 If $(L^{'},[\cdot,\cdot]_L',\beta)$ be another hom-Lie superalgebra, Two quadratic quadratic hom-Lie superalgebras $(L,f,\alpha)$ and
  $(L^{'},f',\beta)$ are said to be isometric if there exists a hom-Lie superalgebras isomorphism $\phi: L\rightarrow L^{'}$ such that
   $f(x, y)=f'(\phi(x), \phi(y)), \forall x, y\in L$.
\end{defn}
\begin{defn}
Let $(L,[\cdot,\cdot]_L,\alpha)$ be a hom-Lie superalgebras over a
field $\K$. We say a $Z_{2}$-graded vector space
$V=V_{\bar{0}}\oplus V_{\bar{1}}$ is a graded $L$-module if
$$L_g\cdot V_h\subseteq V_{g+h},$$
for every $g,h\in Z_{2}$ and
$$[x,y]\cdot v=x\cdot(y\cdot v)-(-1)^{|x||y|}y\cdot(x\cdot v),$$
for every $v\in V, x, y\in L$.
\end{defn}
Now we introduce the cohomology theory of hom-Lie superalgebras, which can be find in \cite{Faouzi&Abdenacer}.

Let $V=V_{\bar{0}}\oplus V_{\bar{1}}$ be a graded $L$-module. Denote
by $C^n(L;V)(n\geq0,C^0(L;V)=V)$ the $Z_{2}$-graded vector space
spanned by all $n$-linear homogenous mappings $f$ of
$L\times\cdots\times L$ into $V$ satisfying
$$f(x_1,\cdots,x_i,x_{i+1},\cdots,x_n)=-(-1)^{|x_i||x_{i+1}|}f(x_1,\cdots,x_{i+1},x_i,\cdots,x_n),$$
where $C^n(L;V)_{\theta}=\{f\in C^n(L;V)| |f(x_1,\cdots,x_n)|=|x_1|+\cdots+|x_n|+\theta\}$, $\theta\in Z_{2}.$

Then $C^n(L;V)$ is a graded $L$-module: If $x\in L$ and $f\in C^n(L;V)$, the action of $L$ on $C^n(L;V)$ is defined by
$$(x\cdot f)(x_1,\cdots,x_n)=x\cdot f(x_1,\cdots,x_n)-\sum_{i=1}^n(-1)^{|x|(|f|+|x_1|+\cdots+|x_{i-1}|)}f(x_1,\cdots,[x,x_i],\cdots,x_n).$$
For a given $r$, we define a map $\delta^{n}_{r}: C^n(L;V)\rightarrow C^{n+1}(L;V)$ by
\begin{eqnarray*}
(\delta^{n}_{r}f)(x_0,\cdots,x_n)&=&\sum_{i=0}^n(-1)^i(-1)^{(|f|+|x_0|+\cdots+|x_{i-1}|)|x_{i}|}[\alpha^{n+r-1}(x_{i}),f(x_0,\cdots,\hat{x_i},\cdots,x_n)]_{V}\\
&+&\sum_{i<j}(-1)^{j+|x_{j}|(|x_{i+1}|+\cdots+|x_{j-1}|)}\\
&&\cdot f(\alpha(x_0),\cdots,\alpha(x_{i-1}),[x_i,x_j],\alpha(x_{i+1}),\cdots,\hat{\alpha(x_j)},\cdots,\alpha(x_n)),
\end{eqnarray*}
Then $\delta^2f=0, \forall f\in C^n(L,V)$. The mapping $f\in C^n(L,V)$ is called an $n$-cocycle if $\delta f=0$. We denote by $Z^{n}(L,V)$ the
subspace spanned by $n$-cocycles. Since $\delta^{2}f=0$ for any $f \in C^{n}(L,V)$, $\delta C^{n-1}(L,V)$ is a subspace of $Z^{n}(L,V)$. Therefore we can define
a cohomology space $H^{n}(L,V)$ of $L$ as the factor space $Z^{n}(L,V)/\delta C^{n-1}(L,V)\,\, (n\geq 0).$

Consider the dual space $L^*$ of $L$, then $L^*$ is a $Z_{2}$-graded
space, where $L_{\theta}^*:=\{\beta\in L^*|\beta(x)=0, \forall
\theta\neq|x|\}$. Moreover, $L^*$ is a $L$-module:
$x\cdot\beta=-(-1)^{|b||\beta|}\beta\ad x$ [10].

The base field $\K$ itself can be considered as $Z_{2}$-graded, if one sets $\K_0=\K$, $\K_{\bar{1}}=\{0\}$. Then as a trivial graded $L$-module,
 $C^n(L,\K)(n\geq0,C^0(L,\K)=\K)$ is the $Z_{2}$-graded vector space spanned by all $n$-linear homogenous mappings $f$ of $L\times\cdots\times L$ into $\K$ satisfying
$$f(x_1,\cdots,x_i,x_{i+1},\cdots,x_n)=-(-1)^{|x_i||x_{i+}|}f(x_1,\cdots,x_{i+1},x_i,\cdots,x_n),$$
where $C^n(L,\K)_{\theta}=\{f\in C^n(L,\K)| f(x_1,\cdots,x_n)=0, \text{if}~ |x_1|+\cdots+|x_n|+\theta\neq0\}$.

Let $(L,[\cdot,\cdot]_L,\alpha)$ be a hom-Lie superalgebras over a field $\K$, $L^{*}=L^{*}_{\bar{0}}\oplus L^{*}_{\bar{1}}$ be its dual space.
Since $L=L_{\bar{0}}\oplus L_{\bar{1}}$ and $L^{*}=L^{*}_{\bar{0}}\oplus L^{*}_{\bar{1}}$ are $Z_{2}$-graded spaces, the direct sum
 $L\oplus L^*=(L_{\bar{0}}\oplus L_{\bar{1}})\oplus(L^{*}_{\bar{0}}\oplus L^{*}_{\bar{1}})=(L_{\bar{0}}\oplus L^{*}_{\bar{0}})\oplus(L_{\bar{1}}\oplus L^{*}_{\bar{1}})$ is $Z_{2}$-graded.

In the sequel, $x+f\in L\oplus L^*$ implies that $x\in L$, $f\in L^*$ and $x+f$ is homogeneous satisfying $|x+f|=|x|=|f|$, then $(-1)^{|x+f||y+g|}=(-1)^{|x||y|}$.
\begin{lem}\cite{Faouzi&Abdenacer}
Let $\ad$ be the adjoint representation of a hom-Lie superalgebra $(L,[\cdot,\cdot]_{L},\alpha)$, and let us consider the even linear map $\pi:L\rightarrow pl(L^{*})$ defined by, $\pi(x)(f)(y)=-(-1)^{|x||f|}f\circ \ad(x)(y),\forall x,y\in L$. Then $\pi$ is a representation of $L$ on $(L^{*},\tilde{\alpha})$ if and only if
\begin{eqnarray}
\ad(x)\circ \ad(\alpha(y))-(-1)^{|x||y|}\ad(y)\circ \ad(\alpha(x))=\alpha\circ \ad([x,y]_{L}).
\end{eqnarray}
We call the representation $\pi$ the coadjoint representation of $L$
\end{lem}
\begin{lem}
Under the above notations, let $(L,[\cdot,\cdot]_L,\alpha)$ be a hom-Lie superalgebra, and $\omega:L\times L\rightarrow L^{*}$ be an even bilinear mapping. Assume that the coadjoint representation exists. The $Z_{2}$-graded spaces $L\oplus L^*$, provided with the following bracket and linear map defined respectively by
\begin{equation}
[x+f,y+g]_{L\oplus L^{\ast}}=[x,y]_{L}+\omega(x,y)+\pi(x)g-(-1)^{|x||y|}\pi(y)f,
\end{equation}
\begin{equation}
\alpha^{'}(x+f)=\alpha(x)+f\circ\alpha.
\end{equation}
then  $(L\oplus L^{*},[\cdot,\cdot]_{L\oplus L^{\ast}},\alpha^{'})$ is a hom-Lie superalgebra if and only if $\omega$ is a 2-cocycle: $L\times L\rightarrow L^{*}$, i.e. $\omega\in Z^2(L, L^*)_{\bar{0}}$.
\end{lem}
\begin{proof}
For any homogeneous elements $x+f, y+g, z+h,\in L\oplus L^{*}$, we obtain
\begin{eqnarray*}
&&-(-1)^{|x||y|}[y+g,x+f]_{L\oplus L^{\ast}}\\
&=&-(-1)^{|x||y|}([y,x]_{L}+\omega(y,x)+\pi(y)f-(-1)^{|x||y|}\pi(x)g)\\
&=&[x,y]_{L}-(-1)^{|x||y|}\omega(y,x)+\pi(x)g-(-1)^{|x||y|}\pi(x)f.
\end{eqnarray*}
Then (1) holds if and only if $\omega\in C^2(L, L^*)_{\bar{0}}$. Note that
\begin{eqnarray*}
&&(-1)^{|x||z|}[\alpha^{'}(x+f),[y+g,z+h]_{L\oplus L^{\ast}}]_{L\oplus L^{\ast}}+c.p.(x+f,y+g,z+h)\\
&=&(-1)^{|x||z|}[\alpha(x),[y,z]_{L}]_{L}+c.p.(x,y,z)\\
&&+(-1)^{|x||z|}\omega(\alpha(x),[y,z]_{L})+(-1)^{|x||z|}\pi(\alpha(x))\omega(y,z)+c.p(x,y,z)\\
&&+(-1)^{|x||z|}\pi(\alpha(x))(\pi(y)h)-(-1)^{|x||z|+|y||z|}\pi(\alpha(x))(\pi(z)g)\\
&&-(-1)^{|x||y|}\pi([y,z]_{L})f\circ\alpha+c.p.(x+f,y+g,z+h).
\end{eqnarray*}
By the hom superJacobi identity , $$(-1)^{|x||z|}[\alpha(x),[y,z]_{L}]_{L}+c.p.(x,y,z)=0.$$
on the other hand
\begin{eqnarray*}
&&(-1)^{|x||z|}\pi(\alpha(x))(\pi(y)h)-(-1)^{|x||z|+|y||z|}\pi(\alpha(x))(\pi(z)g)\\&&-(-1)^{|x||y|}\pi([y,z]_{L})f\circ\alpha+c.p.\\
&=&(-1)^{|y||x|}\pi(\alpha(y))(\pi(z)f)-(-1)^{|z||y|+|x||y|}\pi(\alpha(z))(\pi(y)f)-(-1)^{|x||y|}\pi([y,z]_{L})f\circ\alpha\\
&&+(-1)^{|z||y|}\pi(\alpha(z))(\pi(x)g)-(-1)^{|x||z|+|y||z|}\pi(\alpha(x))(\pi(z)g)-(-1)^{|y||z|}\pi([z,x]_{L})g\circ\alpha\\
&&+(-1)^{|x||z|}\pi(\alpha(x))(\pi(y)h)-(-1)^{|y||x|+|z||x|}\pi(\alpha(y))(\pi(x)h)-(-1)^{|z||x|}\pi([x,y]_{L})h\circ\alpha.
\end{eqnarray*}
Since $\pi$ is the coadjoint representation of $L$, we have
\begin{eqnarray*}
&&(-1)^{|z||x|}\pi([x,y]_{L})h\circ\alpha\\
&=&-(-1)^{|z||y|}h(\alpha\circ\ad([x,y]_{L}))\\
&=&-(-1)^{|z||y|}h\ad(x)\ad(\alpha(y))+(-1)^{|x||y|+|z||y|}h\ad(y) \ad(\alpha(x))\\
&=&(-1)^{|z||y|+|z||x|}(\pi(x)h)(\ad(\alpha(y)))-(-1)^{|x||y|}(\pi(y)h)(\ad(\alpha(x)))\\
&=&-(-1)^{|x||y|+|z||x|}\pi(\alpha(y))(\pi(x)h)+(-1)^{|x||z|}\pi(\alpha(x))(\pi(y)h).
\end{eqnarray*}
Therefore, we have
\begin{eqnarray*}
(-1)^{|x||z|}\pi(\alpha(x))\pi(y)h-(-1)^{|y||x|+|z||x|}\pi(\alpha(y))\pi(x)h-(-1)^{|z||x|}\pi([x,y]_{L})h\circ\alpha=0.
\end{eqnarray*}
Obviously,
$$(-1)^{|x||z|}\pi(\alpha(x))(\pi(y)h)-(-1)^{|x||z|+|y||z|}\pi(\alpha(x))(\pi(z)g)-(-1)^{|x||y|}\pi([y,z]_{L})f\circ\alpha+c.p.=0.$$
Consequently,
$$(-1)^{|x||z|}[\alpha^{'}(x+f),[y+g,z+h]_{\Omega}]_{\Omega}+c.p.(x+f,y+g,z+h)=0.$$
if and only if
\begin{eqnarray*}
0&=&\pi(\alpha(x))(\omega(y,z))-(-1)^{|x||y|}\pi(\alpha(y))(\omega(x,z))+(-1)^{(|x|+|y|)|z|}\pi(\alpha(z))(\omega(x,y))\\
&&+\omega(\alpha(x),[y,z]_{L})+(-1)^{|y||z|}\omega([x,z]_{L},\alpha(y))-\omega([x,y]_{L},\alpha(z)).
\end{eqnarray*}
That is $\omega\in Z^{2}(L,L^{*})$.\\
Then confirmation holds if and only if $\omega\in Z^2(L, L^*)_{\bar{0}}$. Consequently, we prove the lemma.
\end{proof}

Clearly, $L^*$ is an abelian hom-Lie superalgebra ideal of $(L\oplus L^{*},[\cdot,\cdot]_{\alpha^{'}},\alpha^{'})$ and $L$ is isomorphic to the factor hom-Lie superalgebra $(L\oplus L^{*})/L^*$. Moreover, consider the following supersymmetric bilinear form $q_{L}$ on $L\oplus L^{*}$ for all $x+f, y+g\in L\oplus L^*$:
$$q_{L}(x+f,y+g)=f(y)+(-1)^{|x||y|}g(x).$$
Then we have the following lemma:

\begin{lem}\label{lemma3.2}
Let $L$, $L^*$, $\omega$ and $q_L$ be as above. Then the triple $(L\oplus L^{*},q_L,\alpha^{'})$ is a quadratic hom-Lie superalgebra if and only if $\omega$ is supercyclic in the following sense:
$$w(x,y)(z)=(-1)^{|x|(|y|+|z|)}w(y,z)(x) ~~\text{for all}~ x,y,z\in L.$$
\end{lem}
\begin{proof}
The supersymmetric bilinear form $q_{L}$ is nondegenerate:if $x+f$ is orthogonal to  all elements of $L\oplus L^{*}$, then $f(y)=0$ and $(-1)^{|x||y|}g(x)=0$,  which implies that $x=0$ and $f=0$.

Now suppose $x+f,y+g,z+h\in L\oplus L^*$, then we have
\begin{eqnarray*}
&&q_{L}([x+f,y+g]_{L\oplus L^{\ast}}, z+h)\\
&=&q_{L}([x,y]_{L}+\omega(x,y)+\pi(x)g-(-1)^{|x||y|}\pi(y)f, z+h)\\
&=&\omega(x,y)(z)+(\pi(x)g)(z)-(-1)^{|x||y|}(\pi(y)f)(z)+(-1)^{|z|(|x|+|y|)}h([x,y]_{L})\\
&=&\omega(x,y)(z)-(-1)^{|x||y|}g([x,z]_{L})+f([y,z]_{L})+(-1)^{|z|(|x|+|y|)}h([x,y]_{L}).
\end{eqnarray*}
On the other hand,
\begin{eqnarray*}
&&q_{L}(x+f,[y+g,z+h]_{L\oplus L^{\ast}})\\
&=&q_{L}(x+f,[y,z]_{L}+\omega(y,z)+\pi(y)h-(-1)^{|y||z|}\pi(z)g)\\
&=&f([y,z]_{L})+(-1)^{|x|(|y|+|z|)}\omega(y,z)(x)+(-1)^{|x|(|y|+|z|)}(\pi(y)h)(x)\\&&-(-1)^{|x|(|y|+|z|)+|y|z|}(\pi(z)g)(x)\\
&=&f([y,z]_{L})+(-1)^{|x|(|y|+|z|)}\omega(y,z)(x)+(-1)^{|z|(|x|+|y|)}h([x,y]_{L})-(-1)^{|x||y|}g([x,z]_{L}).
\end{eqnarray*}
Hence the lemma follows.
\end{proof}

Now, for a supercyclic 2-cocycle $\omega$ we shall call the quadratic hom-Lie superalgebra $(L\oplus L^{*},q_L,\alpha^{'})$ the $T^*$-extension of $L$ (by $\omega$) and denote the hom-Lie superalgebra  $(L\oplus L^{*},[\cdot,\cdot]_{L\oplus L^{\ast}},\alpha^{'})$ by $T_\omega^*L$.

\begin{defn}
Let $L$ be a hom-Lie superalgebra over a field $\K$. We inductively define a derived series
$$(L^{(n)})_{n\geq 0}: L^{(0)}=L,\ L^{(n+1)}=[L^{(n)},L^{(n)}],$$
a central descending series
$$(L^{n})_{n\geq 0}: L^{0}=L,\ L^{n+1}=[L^{n},L],$$
and a central ascending series
$$(C_{n}(L))_{n\geq 0}: C_{0}(L)=0, C_{n+1}(L)=C(C_{n}(L)),$$
where $C(I)=\{a\in L| [a,L]\subseteq I\}$ for a subspace $I$ of $L$.

$L$ is called solvable and nilpotent(of length $k$) if and only if there is a (smallest) integer $k$ such that $L^{(k)}=0$ and $L^{k}=0$, respectively.
\end{defn}

In the following theorem we discuss some properties of a $T_\omega^*L$.

\begin{thm}
Let $(L,[\cdot,\cdot]_{L},\alpha)$ be a hom-Lie superalgebra over a field $\K$.
\begin{enumerate}[(1)]
   \item  If $L$ is solvable (nilpotent) of length $k$, then the $T^{*}$-extension
          $T^{*}_{\omega}L$ is solvable (nilpotent) of length $r$, where $k\leq r\leq k+1 (k\leq r\leq2k-1)$.
   \item  If $L$ is nilpotent of length $k$, so is the trivial $T^{*}$-extension $T^{*}_{0}L$.
   \item  If $L$ is decomposed into a direct sum of two  hom-Lie superalgebra  ideals of $L$, so is the trivial $T^{*}$-extension $T^{*}_{0}L$.
\end{enumerate}
\end{thm}
\begin{proof}
(1) Suppose first that $L$ is solvable of length $k$. Since
$(T^{*}_{\omega}L)^{(n)}/L^{*}\cong L^{(n)}$ and $L^{(k)}=0$, we have $(T^{*}_{\omega}L)^{(k)}\subseteq L^{*}$, which implies $(T^{*}_{\omega}L)^{(k+1)}=0$ because $L^{*}$ is abelian, and it follows that $T^{*}_{\omega}L$ is solvable of length $k$ or $k+1$.

Suppose now that $L$ is nilpotent of length $k$. Since $(T^{*}_{\omega}L)^{n}/L^{*}\cong L^{n}$ and $L^{k}=0$, we have
$(T^{*}_{\omega}L)^{k}\subseteq L^{*}$. Let $\beta\in(T^{*}_{\omega}L)^{k}\subseteq L^{*}, b\in L$, $x_{1}+f_{1}, \cdots, x_{k-1}+f_{k-1}\in T^{*}_{\omega}L$, $ 1\leq i\leq k-1$, we have
\begin{eqnarray*}
&&[[\!\cdot\!\!\cdot\!\!\cdot\![\beta,x_1+f_{1}]_{L\oplus L^{\ast}},\cdot\!\!\cdot\!\!\cdot]_{L\oplus L^{\ast}},x_{k-1}+f_{k-1}]_{L\oplus L^{\ast}}(b)\\
&=&\beta\ad x_1\cdots\ad x_{k-1}(b)=\beta([x_1,[\cdot\!\!\cdot\!\!\cdot,[x_{k-1},b]_{L}\!\!\cdot\!\!\cdot\!\cdot\!]_{L}]_{L})\in\beta(L^{k})=0.
\end{eqnarray*}
This proves that $(T^{*}_{\omega}L)^{2k-1}=0$. Hence $T^{*}_{w}L$ is nilpotent of length at least $k$ and at most $2k-1$.

(2) Suppose that $L$ is nilpotent of length $k$. Adopting the notations of the proof of part (1), for $x_{k}+f_{k}\in T^{*}_{0}L$, one has
{\setlength\arraycolsep{2pt}
\begin{eqnarray*}
&&[x_1+f_1,[\cdot\!\!\cdot\!\!\cdot,[x_{k-1}+f_{k-1},x_k+f_{k}]_{L\oplus L^{\ast}}\!\!\cdot\!\!\cdot\!\cdot\!]_{L\oplus L^{\ast}}]_{L\oplus L^{\ast}}\\
&=&[x_1,[\cdot\!\!\cdot\!\!\cdot,[x_{k-1},x_k]_{L}\!\!\cdot\!\!\cdot\!\cdot\!]_{L}]_{L}+\sum_{i=1}^k[x_1,[\cdot\!\!\cdot\!\!\cdot,[x_{i-1},[f_i,[x_{i+1},[\cdot\!\!\cdot\!\!\cdot,[x_{k-1},x_k]\!\!\cdot\!\!\cdot\!\cdot\!]]]]\!\!\cdot\!\!\cdot\!\cdot\!]]\\
&=&\ad x_1\cdots\ad x_{k-1}(x_k)+f_1[\ad x_2,[\cdot\!\!\cdot\!\!\cdot,[\ad x_{k-1},\ad x_k]\!\!\cdot\!\!\cdot\!\cdot\!]]\\ &&+(-1)^{k-1}\prod_{i=1}^{k-1}(-1)^{|x_i|(|x_{i+1}|+\cdots+|x_k|)}f_k\ad x_{k-1}\cdots\ad x_1\\
&&+(-1)^{k-2}\prod_{i=1}^{k-2}(-1)^{|x_i|(|x_{i+1}|+\cdots+|x_k|)}f_{k-1}\ad x_{k}\ad x_{k-2}\cdots\ad x_1\\
&&+\sum_{i=2}^{k-2}\prod_{j=1}^{i-1}(-1)^{i-1}(-1)^{|x_j|(|x_{j+1}|+\cdots+|x_k|)}f_i[\ad x_{i+1},[\cdot\!\!\cdot\!\!\cdot,[\ad x_{k-1}, \ad x_k]\!\!\cdot\!\!\cdot\!\cdot\!]]\ad x_{i-1}\cdots\ad x_1,
\end{eqnarray*}}
where we use the fact that $\ad[x,y]=[\ad x, \ad y], \forall x,y\in L$. Note that
$$\ad x_1\cdots\ad x_{k-1}(x_k)\in L^k=0,$$
$$f_1[\ad x_2,[\cdot\!\!\cdot\!\!\cdot,[\ad x_{k-1},\ad x_k]\!\!\cdot\!\!\cdot\!\cdot\!](L)\subseteq f_1(L^k)=0,$$
$$f_k\ad x_{k-1}\cdots\ad x_1(L)\subseteq \alpha_k(L^k)=0,$$
$$f_{k-1}\ad x_{k}\ad x_{k-2}\cdots\ad x_1(L)\subseteq f_{k-1}(L^k)=0,$$
and
$$f_i[\ad x_{i+1},[\cdot\!\!\cdot\!\!\cdot,[\ad x_{k-1}, \ad x_k]\!\!\cdot\!\!\cdot\!\cdot\!]]\ad x_{i-1}\cdots\ad x_1(L)\subseteq f_i(f^k)=0,$$
then the right hand side of the equation vanishes and hence $(T^{*}_{0}L)^k=0$.

(3) Suppose that $0\neq L=I\oplus J$,  where $I$ and $J$ are two nonzero hom-Lie superalgebra ideals of $(L[\cdot,\cdot]_{L},\alpha)$. Let $I^{*}$(resp. $J^{*}$) denote the subspace of all
linear forms in $L^{*}$  vanishing on $J$(resp. $I$). Clearly, $I^{*}$(resp. $J^{*}$) can canonically be identified with the dual space of $I$(resp. $J$) and $L^*\cong I^*\oplus J^*$.

Since $[I^*,L]_{L\oplus L^{\ast}}(J)=I^*([L,J]_{L})\subseteq I^*(J)=0$ and $[I,L^*]_{L\oplus L^{\ast}}(J)=L^*([I,J]_{L})\subseteq L^*(I\cap J)=0$,  we have $[I^*,L]_{L\oplus L^{\ast}}\subseteq I^*$ and $[I,L^*]_{L\oplus L^{\ast}}\subseteq I^*$. Then
\begin{eqnarray*}[T^{*}_{0}I,T^{*}_{0}L]_{L\oplus L^{\ast}}&=&[I\oplus I^*,L\oplus L^*]_{L\oplus L^{\ast}}\\
&=&[I,L]_{L}+[I,L^*]_{L\oplus L^{\ast}}+[I^*,L]_{L\oplus L^{\ast}}+[I^*,L^*]_{L\oplus L^{\ast}}\subseteq I\oplus I^*=T^{*}_{0}I.
\end{eqnarray*}
It is clear that $T^{*}_{0}I$ is a $Z_{2}$-graded space, then $T^{*}_{0}I$ is a hom-Lie superalgebra ideal of $L$ and so is $T^{*}_{0}J$ in the same way. Hence $T^{*}_{0}L$ can be decomposed into the direct sum $T^{*}_{0}I\oplus T^{*}_{0}J$ of two nonzero hom-Lie superalgebra ideals of $T^{*}_{0}L$.
\end{proof}

In the proof of a criterion for recognizing $T^*$-extensions of a hom-Lie superalgebra, we will need the following result.

\begin{lem}\label{lemma3.1}
Let $(L,q_{L},\alpha)$ be a quadratic hom-Lie superalgebra of even dimension $n$ over a field $\K$ and $I$ be an isotropic $n/2$-dimensional subspace of $L$. Then $I$ is a hom-Lie superalgebra of $(L,[\cdot,\cdot]_{L},\alpha)$ if and only if $I$ is abelian.
\end{lem}
\begin{proof}
Since dim$I$+dim$I^{\bot}=n/2+\dim I^{\bot}=n$ and $I\subseteq I^{\bot}$, we have $I=I^{\bot}$.

If $I$ is a hom-Lie superalgebra ideal of $(L,[\cdot,\cdot]_{L},\alpha)$, then $q_{L}(L,[I,I^{\bot}])=q_{L}([L,I],I^{\bot})\subseteq q_{L}(I,I^{\bot})=0$, which implies $[I,I]=[I,I^{\bot}]\subseteq L^{\bot}=0$.

Conversely, if $[I,I]=0$, then $f(I,[I,L])=f([I,I],L)=0$. Hence
$[I,L]\subseteq I^{\bot}=I$. This implies that $I$ is an ideal of
$(L,[\cdot,\cdot]_{L},\alpha)$.
\end{proof}

\begin{thm}
Let $(L,q_{L},\alpha)$ be a quadratic hom-Lie superalgebra of even dimension $n$ over a field $\K$ of characteristic not equal to two. Then $(L,q_{L},\alpha)$ is isometric to a $T^{*}$-extension $(T_{\omega}^{*}B,q_{B},\beta^{'})$ if and only if $n$ is even and $(L,[\cdot,\cdot]_{L},\alpha)$ contains an isotropic hom-Lie superalgebra ideal $I$ of dimension $n/2$. In particular, $B\cong L/I$.
\end{thm}
\begin{proof}
($\Longrightarrow$) Since dim$B$=dim$B^{*}$, dim$T^{*}_{\omega}B$ is even. Moreover, it is clear that $B^{*}$ is a hom-Lie superalgebra ideal of half the dimension of $T^{*}_{\omega}B$ and by the definition of $q_{B}$, we have $q_B(B^*,B^*)=0$, i.e., $B^*\subseteq (B^*)^\bot$ and so $B^*$ is isotropic.

($\Longleftarrow$) Suppose that $I$ is an $n/2$-dimensional isotropic hom-Lie superalgebra ideal of $L$. By Lemma \ref{lemma3.1},  $I$ is abelian. Let $B=L/I$ and $p: L \rightarrow B$ be the canonical projection. Clearly, $|p(x)|=|x|, \forall x\in L$. Since $\ch \K\neq2$,  we can choose an isotropic complement subspace $B_{0}$ to $I$ in $L$, i.e., $L=B_{0}\dotplus I$ and $B_{0}\subseteq B_{0}^{\bot}$. Then $B_{0}^{\bot}=B_{0}$ since dim$B_0=n/2$.

Denote by $p_{0}$ (resp. $p_{1}$) the projection $L \rightarrow B_{0}$ (resp. $L\rightarrow I$) and let $q_{L}^{*}$ denote the homogeneous linear mapping $I \rightarrow B^{*}: i \mapsto q_{L}^{*}(i)$, where $q_{L}^{*}(i)(p(x)):= q_{L}(i,x), \forall x\in L$. We claim that $q_{L}^{*}$ is a linear isomorphism. In fact, if $p(x)=p(y)$, then $x-y\in I$, hence $q_{L}(i,x-y)\in q_{L}(I,I)=0$ and so $q_{L}(i,x)=q_{L}(i,y)$, which implies $q_{L}^{*}$ is well-defined and it is easily seen that $q_{L}^{*}$ is linear. If $q_{L}^{*}(i)=q_{L}^{*}(j)$, then $q_{L}^{*}(i)(p(x))=q_{L}^{*}(j)(p(x)), \forall x\in L$, i.e., $q_{L}(i,x)=q_{L}(j,x)$, which implies $i-j\in L^\bot=0$, hence $q_{L}^{*}$ is injective. Note that $\dim I=\dim B^*$, then $q_{L}^{*}$ is surjective.

In addition, $q_{L}^{*}$ has the following property:
\begin{eqnarray*}
q_{L}^{*}([x,i])(p(y))&=&q_{L}([x,i]_{L},y)=-(-1)^{|x||i|}q_{L}([i,x]_{L},y)=-(-1)^{|x||i|}q_{L}(i,[x,y]_{L})\\
&=&-(-1)^{|x||i|}q_{L}^{*}(i)p([x,y]_{L})=-(-1)^{|x||i|}q_{L}^{*}(i)[p(x),p(y)]_{L}\\
&=&-(-1)^{|x||i|}q_{L}^{*}(i)(\ad p(x)(p(y)))=(\pi(p(x))q_{L}^{*}(i))(p(y))\\
&=&[p(x),q_{L}^{*}(i)]_{L\oplus L^{*}}(p(y)).
\end{eqnarray*}
where $x,y\in L$, $i\in I$. A similar computation shows that
$$q_{L}^{*}([x,i])=[p(x),q_{L}^{*}(i)]_{L\oplus L^{*}},\quad q_{L}^{*}([i,x])=[q_{L}^{*}(i),p(x)]_{L\oplus L^{*}}.$$

Define a homogeneous bilinear mapping
\begin{eqnarray*}
\omega:~~~~~ B\times B~~~~&\longrightarrow&B^{*}\\
(p(b_0),p(b_0'))&\longmapsto&q_{L}^{*}(p_{1}([b_0,b_0'])),
\end{eqnarray*}
where $b_0,b_0'\in B_{0}.$ Then $|w|=0$ and $w$ is well-defined since the restriction of the projection $p$ to $B_{0}$ is a linear isomorphism.

Now, define the bracket on $B\oplus B^{*}$ by (12), we have $B\oplus B^{*}$ is a $Z_{2}$-graded algebra. Let $\varphi$ be the linear mapping $L \rightarrow B\oplus B^{*}$ defined by $\varphi(b_{0}+i)=p(b_{0})+q_{L}^{*}(i), \forall b_0+i\in B_0\dotplus I=L. $
 Since the  restriction of $p$  to $B_{0}$ and $q_{L}^{*}$ are linear isomorphisms, $\varphi$ is also a linear isomorphism. Note that
\begin{eqnarray*}
&&\varphi([b_0+i,b_0'+i']_{L})=\varphi([b_0,b_0']_{L}+[b_0,i']_{L}+[i,b_0']_{L})\\
&=&\varphi(p_{0}([b_0,b_0']_{L})+p_{1}([b_0,b_0']_{L})+[b_0,i']_{L}+[i,b_0']_{L})\\
&=&p(p_{0}([b_0,b_0']_{L}))+q_{L}^{*}(p_{1}([b_0,b_0']_{L})+[b_0,i']_{L}+[i,b_0']_{L})\\
&=&[p(b_0),p(b_0')]_{L}+\omega(p(b_0),p(b_0'))+[p(b_0),q_{L}^{*}(i')]_{L}+[q_{L}^{*}(i),p(b_0')]_{L}\\
&=&[p(b_0),p(b_0')]_{L}+\omega(p(b_0),p(b_0'))+\pi(p(b_0)(q_{L}^{*}(i'))-(-1)^{|b_0||b_0'|}\pi(p(b_0')(q_{L}^{*}(i))\\
&=&[p(b_0)+q_{L}^{*}(i),p(b_0')+q_{L}^{*}(i')]_{B\oplus B^{*}}\\
&=&[\varphi(b_0+i),\varphi(b_0'+i')]_{L\oplus L^{*}}.
\end{eqnarray*}
Then $\varphi$ is an isomorphism of $Z_{2}$-graded algebras, and so $(B\oplus B^{*},[\cdot,\cdot]_{B\oplus B^{*}},\beta)$ is a hom-Lie superalgebra.
Furthermore, we have
{\setlength\arraycolsep{2pt}
\begin{eqnarray*}
q_{B}(\varphi(b_{0}+i),\varphi(b_{0}'+i'))&=&q_{B}(p(b_{0})+q_{L}^{*}(i),p(b_{0}')+q_{L}^{*}(i'))\\
&=&q_{L}^{*}(i)(p(b_{0}'))+(-1)^{|b_0||b_0'|}q_{L}^{*}(i')(p(b_{0}))\\
&=&q_{L}(i,b_{0}')+(-1)^{|b_0||b_0'|}q_{L}(i',b_{0})\\
&=&q_{L}(b_{0}+i,b_{0}'+i'),
\end{eqnarray*}}
then $\varphi$ is isometric. The relation
\begin{eqnarray*}
q_{B}([\varphi(x),\varphi(y)],\varphi(z))&=&q_B(\varphi([x,y]),\varphi(z))\\
&=&q_{L}([x,y],z)=q_{L}(x,[y,z])=q_{B}(\varphi(x),[\varphi(y),\varphi(z)])
\end{eqnarray*}
implies that $q_B$ is a nondegenerate invariant supersymmetric bilinear form, and so $(B\oplus B^{*}, q_B,\beta^{'})$ is a quadratic hom-Lie superalgebra.
In this way, we get a $T^*$-extension $T^{*}_{\omega}B$ of $B$ and consequently, $(L,q_{L},\alpha)$ and $(T^{*}_{\omega}B,q_{B},\beta^{'})$ are isometric as required.
\end{proof}

The proof of Theorem 4.10 shows that the homogeneous bilinear
mapping $\omega$ depends on the choice of the isotropic subspace
$B_0$ of $L$ complement to the hom-Lie superalgebra ideal $I$.
Therefore there may be different $T^*$-extensions describing the
``same'' quadratic hom-Lie superalgebras.

Let$(L,[\cdot,\cdot]_{L},\alpha)$ be a hom-Lie superalgebra over a field $\K$. and let $\omega_{1}: L\times L\rightarrow L^{*}$ and $\omega_{2}: L \times L\rightarrow L^{*}$ be two different supercyclic 2-cocycles satisfying $|\omega_1|=|\omega_2|=0$. The $T^{*}$-extensions $T^{*}_{\omega_{1}}L$ and $T^{*}_{w_{2}}L$ of $L$ are said to be {\bf equivalent}  if there exists an isomorphism of hom-Lie superalgebra $\phi: T^{*}_{\omega_{1}}L\rightarrow  T^{*}_{\omega_{2}}L$ which is the identity on the hom-Lie superalgebra ideal $L^{*}$ and which induces the identity on the factor hom-Lie superalgebra algebra $T^{*}_{\omega_{1}}L/L^{*}\cong L\cong T^{*}_{\omega_{2}}L/L^{*}.$ The two $T^{*}$-extensions $T^{*}_{\omega_{1}}L$ and $T^{*}_{\omega_{2}}L$ are said to be {\bf isometrically equivalent} if  they are equivalent and $\phi$ is an isometry.
\begin{prop}\label{proposition3.1}
Let $L$ be a hom-Lie superalgebra over a field $\K$ of characteristic not equal to 2, and $\omega_{1}$, $\omega_{2}$ be two homcyclic 2-cocycles $L\times L\rightarrow L^{*}$ satisfying $|\omega_i|=0$. Then we have
\begin{enumerate}[(i)]
   \item  $T^{*}_{\omega_{1}}L$ is equivalent to  $T^{*}_{\omega_{2}}L$ if and only if there is $z\in C^1(L, L^{*})_{\bar{0}}$ such that
    \begin{equation}\label{equation3.1}
       \omega_{1}(x, y)-\omega_{2}(x, y)=\pi(x)z(y)-(-1)^{|x||y|}\pi(y)z(x)-z([x,y]_{L}),  \forall x, y\in L.
    \end{equation}

        If this is the case, then the supersymmetric part $z_{s}$ of $z$, defined by $z_{s}(x)(y):=\frac{1}{2}(z(x)(y)+(-1)^{|x||y|}z(y)(x))$,
        for all $x,y\in L$, induces a supersymmetric invariant bilinear form on $L$.
   \item  $T_{\omega_{1}}^{*}L$ is isometrically equivalent to $T_{\omega_{2}}^{*}L$ if and only if there is $z\in C^1(L, L^{*})_{\bar{0}}$ such that $(\ref{equation3.1})$ holds for all ~$x,y\in L$ and the supersymmetric part $z_{s}$ of $z$ vanishes.
\end{enumerate}
\end{prop}
\begin{proof}
(i) $T^{*}_{\omega_{1}}L$ is equivalent to  $T^{*}_{\omega_{2}}L$ if and only if there is an isomorphism of hom-Lie superalgebra $\Phi: T_{\omega_{1}}^{*}L\rightarrow T_{\omega_{2}}^{*}L$ satisfying $\Phi|_{L^*}=1_{L^*}$ and $x-\Phi(x)\in L^*, \forall x\in L$.

Suppose that $\Phi: T_{\omega_{1}}^{*}L\rightarrow T_{\omega_{2}}^{*}L$ is an isomorphism of hom-Lie superalgebra and define a linear mapping $z: L\rightarrow L^{*}$ by $z(x):=\Phi(x)-x$, then $z\in C^1(L, L^{*})_{\bar{0}}$ and for all $x+f,y+g\in T^{*}_{\omega_{1}}L$, we have
\begin{eqnarray*}
&&\Phi([x+f,y+g]_{\Omega})\\
&=&\Phi([x,y]_{L}+\omega_{1}(x,y)+\pi(x)g-(-1)^{|x||y|}\pi(y)f)\\
&=&[x,y]_{L}+z([x,y]_{L})+\omega_{1}(x,y)+\pi(x)g-(-1)^{|x||y|}\pi(y)f.
\end{eqnarray*}
On the other hand,
\begin{eqnarray*}
&&[\Phi(x+f),\Phi(y+g)]\\
&=&[x+z(x)+f,y+z(y)+g]\\
&=&[x,y]_{L}+\omega_2(x,y)+\pi(x)g+\pi(x)z(y)-(-1)^{|x||y|}\pi(y)z(x)-(-1)^{|x||y|}\pi(y)f.
\end{eqnarray*}
Since $\Phi$ is an isomorphism, (14) holds.

Conversely, if there exists $z\in C^{1}(L, L^{*})_{\bar{0}}$ satisfying (12), then we can define $\Phi: T_{\omega_{1}}^{*}L\rightarrow T_{\omega_{2}}^{*}L$ by $\Phi(x+f):=x+z(x)+f$. It is easy to prove that $\Phi$ is an isomorphism of hom-Lie superalgebra such that $\Phi|_{L^*}=1_{L^*}$ and $x-\Phi(L)\in L^*, \forall x\in L$, i.e., $T^{*}_{\omega_{1}}L$ is equivalent to  $T^{*}_{\omega_{2}}L$.

Consider the supersymmetric bilinear form $q_{L}: L\times L\rightarrow \K, (x,y)\mapsto z_s(x)(y)$ induced by $z_s$. Note that
\begin{eqnarray*}
&&\omega_{1}(x, y)(m)-\omega_{2}(x, y)(m)\\
&=&\pi(x)z(y)(m)-(-1)^{|x||y|}\pi(y)z(x)(m)-z([x,y]_{L})(m)\\
&=&(-1)^{|x||y|}z(y)([x,m]_{L})+z(x)([y,m]_{L})-z([x,y]_{L})(m)
\end{eqnarray*}
and
\begin{eqnarray*}
&&(-1)^{|x|(|y|+|m|)}(\omega_{1}(y,m)(x)-\omega_{2}(y,m)(x))\\
&=&(-1)^{|x|(|y|+|m|)}(\pi(y)z(m)(x)-(-1)^{|y||m|}\pi(m)z(y)(x)-z([y,m]_{L})(x))\\
&=&(-1)^{|m|(|x|+|y|)}z(m)([x,y]_{L})+(-1)^{|x||y|}z(y)([x,m]_{L})-(-1)^{|x|(|y|+|m|)}z([y,m]_{L})(x)
\end{eqnarray*}
Since both $\omega_{1}$ and $\omega_{2}$ are homcylic, the right hand sides of above two equations are equal. Hence
\begin{eqnarray*}
&&(-1)^{|x||y|}z(y)([x,m]_{L})+z(x)([y,m]_{L})-z([x,y]_{L})(m)\\
&=&(-1)^{|m|(|x|+|y|)}z(m)([x,y]_{L})+(-1)^{|x||y|}z(y)([x,m]_{L})-(-1)^{|x|(|y|+|m|)}z([y,m]_{L})(x).
\end{eqnarray*}
That is
\begin{eqnarray*}
z(x)([y,m]_{L})+(-1)^{|x|(|y|+|m|)}z([y,m]_{L})(x)=z([x,y]_{L})(m)+(-1)^{|m|(|x|+|y|)}z(m)([x,y]_{L}).
\end{eqnarray*}
Since ch$\K\neq2$, $q_{L}(x,[y,m])=q_{L}([x,y],m)$, which proves the invariance of the supersymmetric bilinear form $q_{L}$ induced by $z_s$.

(ii) Let the isomorphism $\Phi$ be defined as in (i). Then for all $x+f, y+g\in L\oplus L^{*}$, we have
\begin{eqnarray*}
&&q_{B}(\Phi(x+f),\Phi(y+g))=q_{B}(x+z(x)+f,y+z(y)+g)\\
&=&z(x)(y)+f(y)+(-1)^{|x||y|}(z(y)(x)+g(x))\\
&=&z(x)(y)+(-1)^{|x||y|}z(y)(x)+f(y)+(-1)^{|x||y|}g(x)\\
&=&2z_s(x)(y)+q_{B}(x+f, y+g).
\end{eqnarray*}
Thus $\Phi$ is an isometry if and only if $z_{s}=0$.
\end{proof}


\begin{thebibliography}{99}

\bibitem {1} Ammar F.  and  Makhlouf A., \textsl{Hom-Lie superalgebras and Hom-Lie admissible
superalgebras.} J. Algebra 324 (2010), no. 7, 1513-1528.

\bibitem {Faouzi&Abdenacer} Ammar F.  and  Makhlouf A., \textsl{Cohomology of Hom-Lie superalgebras and q- deforemed Witt superalgebra}, e-Print: arXiv:1204.6244.

\bibitem {Faouzi&Abdenacer} Ammar F., Ayadi I.  and  Makhlouf A., \textsl{Quadratic color Hom-Lie
algebras}, e-Print: arXiv:1204.5155.

\bibitem {Bajo&Benayadi&Medina} Bajo I., Benayadi S., Medina A., \textsl{Symplectic structures on quadratic Lie algebras.} J. Algebra 316(2007), no. 1, 174-188.

\bibitem {Bordemann} Bordemann M. \textsl{Nondegenerate invariant bilinear forms on nonassociative algebras.} Acta Math. Univ. Comenianae LXVI(2)(1997), 151-201.


\bibitem {1} Gohr A. On hom-algebras with surjective twisting. J. Algebra 324 (2010), no. 7, 1483-1491.


\bibitem {Hartwig&Larsson}  Hartwig J., Larsson D.  and Silvestrov S., \textsl{Deformations of Lie algebras using $\sigma$-derivations,}  J. Algebra 295(2006), 321-344.

\bibitem {1} Jin Q. and Li X.,  Hom-Lie algebra structures on semi-simple Lie algebras. J. Algebra 319 (2008), no. 4, 1398-1408.

\bibitem {} Lin J., Wang Y.,  Deng S.,\textsl{$T^{*}$-extension of Lie triple systems}. Linear Algebra and its Applications {431}(2009), 2071-2083.

\bibitem {Scheunert&Zhang} Scheunert M., Zhang R., \textsl{Cohomology of Lie superalgebras and their generalizations.} J. Math. Phys. 39(1998), 5024-5061.

\bibitem {Sheng&Yunhe} Sheng Y., \textsl{Representations of Hom-Lie Algebras.} Algebr. Represent. Theory, 15 (2012), no. 6, 1081-1098.


\bibitem {Yau&Donald} Yau D., The Hom-Yang-Baxter equation and Hom-Lie algebras. J. Math. Phys. 52 (2011), no. 5, 053502, 19 pp.



\bibitem {1}  Yau D., Hom-quantum groups: I. Quasi-triangular Hom-bialgebras. J. Phys. A 45 (2012), no. 6, 065203, 23 pp.



\end{thebibliography}
\end{document}